\documentclass[letterpaper]{amsart}
\usepackage{amssymb,amsmath,mathrsfs,tikz}
\usetikzlibrary{matrix,arrows}
\usepackage{euler}
\newtheorem{theorem}{Theorem}[section]
\newtheorem{lemma}[theorem]{Lemma}
\newtheorem{corollary}[theorem]{Corollary}
\theoremstyle{definition}
\newtheorem{definition}[theorem]{Definition}

\theoremstyle{remark}
\newtheorem{remark}[theorem]{Remark}

\newenvironment{ack}{\noindent {\bf Acknowledgements}}{}{}%
\def\posets{\operatorname{Pos}}

\def\TT{\mathcal T}
\def\FF{\mathcal F}
\newcommand{\leqtope}[1]{\preccurlyeq_{#1}}

\newcommand{\pos}[2]{[#2, #1]}

\newcommand{\kk}{\ensuremath{\Bbbk}}

\newcommand{\Z}{\ensuremath{\mathbb Z}}
\newcommand{\C}{\ensuremath{\mathbb C}}
\newcommand{\R}{\ensuremath{\mathbb R}}
\newcommand{\RR}{\ensuremath{\mathcal R}}
\newcommand{\FFF}{\ensuremath{\mathsf M}}

\newcommand{\A}{\ensuremath{\mathscr A}}

\newcommand{\W}{\ensuremath{\underline{\mathsf W}}}

\newcommand{\QQ}{\ensuremath{\mathcal Q}}
\newcommand{\PP}{\ensuremath{\mathcal P}}
\renewcommand{\SS}{\ensuremath{\mathcal S}}
\newcommand{\uF}{\ensuremath{\underline{\FFF}}}
\newcommand{\DD}{\ensuremath{\mathcal D}}
\newcommand{\CC}{\ensuremath{\mathscr C}}

\newcommand{\supp}{\operatorname{supp}}

\begin{document}
\title[Equivariant model]{An equivariant discrete model for complexified arrangement complements} %
\author{Emanuele Delucchi and Michael J. Falk}
\address{Fachbereich Mathematik und Informatik, Universit\"at Bremen,
  Bibliothekstra\ss{}e 1, 20135 Bremen, Germany.}
\address{Department of Mathematics and Statistics,
Northern Arizona University,
Flagstaff, AZ 86011-5717, USA.}
\maketitle


\begin{abstract}
We define a partial ordering on the set $\QQ=\QQ(\FFF)$ of pairs of topes of an oriented matroid \FFF, and show the geometric realization $|\QQ|$ of the order complex of \QQ\  has the same homotopy type as the Salvetti complex of \FFF. For any element $e$ of the ground set, the complex $|\QQ_e|$ associated to the rank-one oriented matroid on $\{e\}$ has the homotopy type of the circle. There is a natural free simplicial action of $\Z_4$ on $|\QQ|$, with orbit space isomorphic to the order complex of the poset $\QQ(\FFF,e)$ associated to the pointed (or affine) oriented matroid $(\FFF,e)$. If \FFF\ is the oriented matroid of an arrangement \A\ of linear hyperplanes in $\R^n$, the $\Z_4$ action corresponds to the diagonal action of $\C^*$ on the complement $M$ of the complexification of \A: $|\QQ|$ is equivariantly homotopy-equivalent to $M$ under the identification of $\Z_4$ with $\{\pm1, \pm i\}$, and $|\QQ(\FFF,e)|$ is homotopy-equivalent to the complement of the decone of \A\ relative to the hyperplane corresponding to $e$. All constructions and arguments are carried out at the level of the underlying posets. If a group $G$ acts on the set of topes of \FFF\ preserving adjacency, then $G$ acts simplicially on $|\QQ|$.

We also show that the class of fundamental groups of such complexes is strictly larger than the class of fundamental groups of complements of complex hyperplane arrangements. Specifically, the group of the non-Pappus arrangement is not isomorphic to any realizable arrangement group.
\end{abstract}

\section{Introduction}
An arrangement of hyperplanes is a set $\mathscr A =
\{H_1,\ldots,H_n\}$ of linear or affine codimension $1$ subspaces of $\mathbb
C^d$. An arrangement is {\em complexified} if each $H_i$ has a defining equation with real coefficients; in this case the underlying real arrangement $\{H_1\cap \R^d, \ldots, H_n \cap \R^d \}$ is denoted $\A_\R$.
A main topic in the theory of hyperplane arrangements is the study of
combinatorial invariants of the topology of the complement $M(\mathscr
A):=\mathbb C^d \setminus \bigcup \mathscr A$.

The arrangement $\mathscr A $ is called {\em central} if all its 
hyperplanes contain the origin; in this case, $M(\mathscr A)$ carries the
natural (diagonal) $\mathbb C^*$-action. One of the many consequences
of this fact is the following topological property.
Fix an element $H_0\in
\mathscr A$ and let $H_0'$ be a parallel translate of $H_0$ that does not
contain the origin. Let $d\mathscr A$ be the {\em decone} of $\mathscr
A$ relative to $H_0$, the arrangement $\{H\cap H_0' \mid H \in \mathscr
A\setminus H_0\}$ in $H_0'\cong \mathbb C^{d-1}$. Then there is a
diffeomorphism
$$M(\mathscr A) \cong \mathbb C^*\times M(d\mathscr A).$$

There exist combinatorially defined complexes that model the homotopy
type of $M(\mathscr A)$, e.g.,  by work of Salvetti
\cite{Salvetti} in the complexified case, and Bj\"orner and Ziegler
\cite{BZ} in the general case. These complexes are finite, therefore
cannot model the circle action of $S^1\subset \mathbb C^*$ on
$M(\mathscr A)$. 

In principle, there are two ways out of this
situation: either to develop `continuous' combinatorial models that can carry
a circle action, or to let a `discretized' $S^1$ act on the known
combinatorial models. A continuous approach has been
attempted, e.g.\ in \cite{AnDe}, and is as yet not fully developed. Here
we explore the second possibility, also in view of the fact that the
simplicial complexes mentioned above are defined in the
general setting of {\em pseudosphere arrangements}, where no 
original linear space with $\mathbb C^*$ action exists.

The known discrete complexes depend only on the combinatorics of arrangements of real codimension-one pseudo-spheres in $S^{d-1}$, encoded by the associated oriented matroid or $2$-matroid,  respectively, and are defined as the {\em order complexes} of certain partially-ordered sets, or posets. The order complex of a poset \PP\ is the abstract simplicial complex $\Delta(\PP)$ whose simplices are the linearly-ordered subsets, or {\em chains}, of \PP.  Order-preserving and order-reversing maps of posets induce simplicial maps of order complexes. The geometric realization of $\Delta(\PP)$ is denoted $|\PP|$, and is called the {\em geometric realization} of \PP\ (see Remark~\ref{rem:vary}). 

Here we treat only complexified arrangements, in the general setting
of oriented matroids. Associated to a loop-free oriented matroid \FFF,
one has the {\em Salvetti poset} $\SS=\SS(\FFF)$ whose geometric
realization $|\SS|$ has the homotopy type of $M$ in case \FFF\ is
realized by the real arrangement $\A_\R$ while in general, by a result of Deshpande
\cite{PD}, $|\SS|$ has the
homotopy type of the tangent bundle complement of the 
arrangement of pseuduspheres associated to $\FFF$ (see Definition \ref{def:pseudoarr}). If $e_0$ is a fixed element of the ground set of \FFF\ (corresponding to $H_0\in\A$) one has the pointed (or affine) oriented matroid $(\FFF,e_0)$, and an associated subposet $d\SS=\SS(\FFF,e_0)$ of $\SS$, with $|d\SS|$ homotopy equivalent to the complement of the decone $d\A$ of \A\ relative to $H_0$.

In this paper, after a preparatory section on the basics of poset topology, we
\begin{itemize}
\item define posets $\QQ=\QQ(\FFF)$ and $d\QQ=\QQ(\FFF,e_0) \subseteq \QQ$ and an order-preserving map $\SS \to \QQ$ inducing homotopy equivalences $|\SS| \simeq |\QQ|$ and $|d\SS| \simeq |d\QQ|$;
\item define a natural free action of $\Z_4$ on \QQ\ by order-reversing and -preserving isomorphisms;
\item define an equivariant order-preserving map $\QQ_{e_0} \times d\QQ \to \QQ$, where $\QQ_{e_0}$ is the poset associated with $\FFF|_{\{e_0\}}$ and $\Z_4$ acts trivially on $d\QQ$, inducing a homotopy equivalence $|\QQ_{e_0}| \times |d\QQ| \simeq |\QQ|$. Then $|\SS_{e_0}| \times |d\SS|\simeq |\SS|$ as well.
\end{itemize}

Thus we obtain a combinatorial version of the cone-decone property of complexified hyperplane arrangements, which holds in the ostensibly more general setting of oriented matroids, realizable or not. As a corollary we obtain the main result of \cite{cordovil}, a product decomposition $\pi_1(|\SS|) \cong \Z \times \pi_1(|d\SS|)$ of fundamental groups, originally proved via complicated manipulation of group presentations. Our work also partly answers a question of Ziegler \cite[Problem 7.7]{Z}.

Finally we show that this setting is indeed more general, by displaying an oriented matroid \FFF, an orientation of the non-Pappus matroid, for which $\pi_1(\vert \SS \vert)$ is not isomorphic to the fundamental group of the complement of any complex hyperplane arrangement. To our knowledge no such example has appeared in the literature. The argument uses properties of degree-one resonance varieties.

\section{Poset topology}

\begin{definition}
  A partially ordered set (or {\em poset}) is a pair $(\PP,\leq)$ where
  $\PP$ is a set and $\leq$ a partial order relation on $\PP$. A morphism of
  posets $(\PP,\leq_\PP)\to (\QQ,\leq_\QQ)$ is an order-preserving function
  $f:\PP\to \QQ$, i.e., one for which $f(p_1)\leq_\QQ f(p_2)$ whenever
  $p_1\leq_\PP p_2$; it is an isomorphism if $f$
  is bijective, and in this case we will write $(\PP,\leq_\PP)\cong(\QQ\leq_\QQ)$. We will write $\posets$ for the category of posets
  and order-preserving functions. A {\em chain} in the poset $(\PP,\leq)$ is a subset of $\PP$ that is
  totally ordered by $\leq_\PP$.  The {\em product} of two posets $(\PP,\leq_\PP)$ and $(\QQ,\leq_\QQ)$ is
  $(\PP\times \QQ, \leq_{\PP\times \QQ})$, where $(p_1,q_1)\leq_{(\PP\times
    \QQ)}(p_2,q_2)$ if and only if $p_1\leq_\PP p_2$ and $q_1\leq_\QQ q_2$.

  The {\em opposite} or `order dual' of a given poset $(\PP,\leq_\PP)$
  is the poset $(\PP,\leq_\PP)^{op}=(\PP,\leq_\PP^{op})$ where $p_1\leq_\PP^{op} p_2$ if and
  only if
  $p_1\geq_{\PP} p_2$.
\end{definition}

\begin{remark}[Notation]
  It is customary to denote a poset $(\PP,\leq)$ by its underlying set
  $\PP$ when the order relation is understood. 
\end{remark}



 Let $\PP$ be a poset. Let $(\Delta(\PP),\leq)$ be the poset of chains in \PP, with $\sigma \leq \tau$ if and only if $\sigma \subseteq \tau$. $\Delta(\PP)$ is an abstract simplicial complex with vertex set $\PP$, called the {\em 
order complex} of $\PP$. The standard geometric realization of $\Delta(\PP)$ will be denoted by $\vert \PP \vert$, and called the {\em geometric realization} of \PP. We refer to \cite{Kozlov} as a general reference for poset topology.

\begin{remark}\label{rem:vary} The terminology leads to no conflict: if \PP\ is a simplicial complex, there is a simplicial homeomorphism of $\vert \Delta(\PP)\vert$ to the barycentric subdivision of $\vert \PP \vert$. See also Remark~\ref{rem:can} below.  
\end{remark}


As is customary, we refer to the
  homotopy type of $\vert \PP \vert$ when speaking of ``the
  homotopy type of the poset $\PP$." In particular, we will say that
  posets $\PP$ and $\QQ$ are homotopy equivalent  
  (written $\PP\simeq \QQ$) if $\vert \PP \vert$ and $\vert \QQ \vert$ are. 

\begin{remark}\label{rem:op} $\,$
  \begin{itemize}
  \item[(a)]   For every poset $\PP$ we have $\Delta(\PP)=\Delta(\PP^{op})$.
   \item[(b)] If $\PP$ and $\QQ$ are posets, then $\vert \PP \times \QQ\vert$ is homeomorphic to
    $\vert \PP\vert\times \vert \QQ \vert$. (In fact $\Delta(\PP \times \QQ)$ is a triangulation of $|\PP| \times |\QQ|$.) See \cite[Theorem 10.21]{Kozlov} for a generalization.
  \end{itemize}
\end{remark}
\noindent
The following ``Quillen Lemma" is used in most applications of the theory - see \cite[Thm. 15.28]{Kozlov}.
\begin{lemma}[See \cite{Q}] \label{lem:quillen}
  Let $f:\PP\to \QQ$ be a poset map. If $f^{-1}(\QQ_{\geq q})$ is
  contractible for all $q\in \QQ$, then $\PP\simeq \QQ$. 
\end{lemma}

\begin{remark}\label{rem:replace}
  The condition of Lemma \ref{lem:quillen} can be replaced by ``$f^{-1}(\QQ_{\leq q})$ is
  contractible for all $q\in \QQ$'' via Remark \ref{rem:op}.
\end{remark}

\begin{definition}
  An order-preserving function $f:\PP\to \PP$ is {\em monotone} if either
  $f(p)\geq p$ for all $p\in \PP$ or $f(p) \leq p$ for all $p\in \PP$.
\end{definition}
\def\im{\operatorname{im}}
\def\fix{\operatorname{\sf fix}}

\begin{lemma}[See Theorem 13.22(b) in \cite{Kozlov}]\label{rem:monotone}
 Let $f:\PP\to \PP$ be a monotone poset map. Then $ \PP \simeq \fix (f)$.  
\end{lemma}

\begin{remark}\label{remCone}
  If a poset $\PP$ has a unique maximal element $p$, then $\PP$ is contractible
  because its order complex is the cone over the order complex of $\PP\setminus \{p\}$.
\end{remark}

\begin{remark}\label{rem:can} For every poset $\PP$, there is a canonical homotopy equivalence $\Delta(\PP)\simeq \PP$ 
 (e.g., by the function 
 $\Delta(\PP)\to \PP$, $\omega \mapsto \min \omega $). 
\end{remark}

\begin{definition}
  If $\PP$ has a unique maximal element (say, $x$), we define
  $\Delta^\dag(\PP)$ to be the subposet of $\Delta(\PP)$ consisting of
  all chains in $\PP$ that contain $x$. 
  If $\PP$ has a unique maximal element as well as a unique minimal element,
  then $\Delta^{\dag\dag}(\PP)$ is the poset of chains containing both.
\end{definition}

\begin{lemma}\label{lem:delta} Let $\PP$ be a poset with a unique maximal element. Then
$\Delta^\dag(\PP) \simeq \Delta(\PP)$. Moreover, if $\PP$ has a unique
minimal element as well, $\Delta^{\dag\dag}(\PP) \simeq \Delta \PP$  
\end{lemma}

\begin{proof} Let $p$ be the unique maximal element of $\PP$.
  The assignment $\omega \mapsto \omega \cup \{p\}$ 
defines a monotone poset map $\Delta(\PP)\to \Delta(\PP)$ fixing
$\Delta^\dag(\PP)$. 

The proof for $\Delta^{\dag\dag}(\PP)$ goes similarly.
\end{proof}

\section{Discrete circle action on complexified arrangements}

For the remaining of this paper fix a rank $r$ oriented matroid on  finite ground
set $E$ and let $\mathcal F$ be its set of covectors. For an
introduction to the theory of oriented matroids see  \cite{BLSWZ}:
here we recall only what is needed in the following.

\begin{definition}\cite[Definition  5.1.3]{BLSWZ}\label{def:pseudoarr}
  A {\em rank-$r$ arrangement of pseudospheres} is a set $\mathscr A
  =\{S_e\}_{e\in E}$
  of centrally symmetric PL-homeomorphic embeddings of $S^{r-2}$ in
  $S^{r-1}$ such that, for all $\mathscr B\subset \mathscr A$,
  $\bigcap \mathscr B$ is a PL-sphere, together with a choice of a connected component $S_e^+$ of
  $S^r\setminus S_e$ for every $e\in E$.

  The set of {\em real signs} is $\{+,0,-\}$, and the map
  $$
  \sigma: S^{r-1} \to \{+,0,-\}^E; \quad \sigma_{\mathscr A}(x)_e:=
  \left\{\begin{array}{ll}
      + & \textrm{if }x\in S_e^+ \\
      0 & \textrm{if }x\in S_e \\
      - & \textrm{else. }
  \end{array}\right.
  $$
  associates a {\em sign vector} to every point of the sphere. Notice
  that the zero vector $\hat{0}:=(0,\ldots,0)$ is not in the image of
  $\sigma_{\mathscr A}$.

  The set of covectors of a {\em rank-$r$  oriented matroid} on the
  ground set $E$ is any
  subset $\FF\subseteq \{+,0,-\}^E$ of the form $\FF=\im(\sigma_{\mathscr
    A})\cup \{\hat{0}\}$ for some rank $r$ arrangement of pseudospheres $\mathscr A$.
\end{definition}

\begin{remark} 
  If we partially order the set of signs $\{+,0,-\}$ by $0<+$, $0<-$
  and $+$ incomparable to $-$,
  the set $\FF$
  inherits a partial order
  $\leq_{\FF}$ as a subset of the
  product poset $\{+,0,-\}^E$. 
  With this partial ordering, $\FF$ has a unique minimal element
 $\hat{0}$ and a set $\TT$ of maximal elements, called {\em topes}.

 Notice that, on $\FF\setminus \{\hat{0}\}$, the ordering $\leq_{\mathcal F}$ coincides with the
 incidence relation of closed cells of the stratification of
 $S^{r-1}$. 
\end{remark}


\begin{figure}[h]
  \begin{tikzpicture}
    \foreach \a in {0,60,120} \draw (\a:2) -- (\a:-2);
    \node (A) at (270:1) {$A$}; \node (B) at (210:1) {$B$}; \node (C)
    at (330:1) {$C$}; \node (D) at (150:1) {$D$}; \node (E) at (30:1)
    {$E$}; \node (F) at (90:1) {$F$};
  \end{tikzpicture}
  \begin{tikzpicture}[scale=0.5]
    \node (A1) at (0,4.5) {$A$};
    \node (B1) at (2,4.5) {$B$};
    \node (D1) at (4,4.5) {$D$};
    \node (F1) at (6,4.5) {$F$};
    \node (E1) at (8,4.5) {$E$};
    \node (C1) at (10,4.5) {$C$};
    \coordinate (A) at (0,4);
    \coordinate (B) at (2,4);
    \coordinate (D) at (4,4);
    \coordinate (F) at (6,4);
    \coordinate (E) at (8,4);
    \coordinate (C) at (10,4);
    \coordinate (a) at (0,2);
    \coordinate (b) at (2,2);
    \coordinate (c) at (4,2);
    \coordinate (d) at (6,2);
    \coordinate (e) at (8,2);
    \coordinate (f) at (10,2);
    \foreach \p in {0,2,...,10} 
        \draw[fill] (\p,2) circle [radius=0.1];
    \foreach \p in {0,2,...,10} 
        \draw (\p,2) -- (5,0);        
    \draw (A) -- (a) -- (B) -- (b) -- (D) -- (c) -- (F) -- (d) -- (E) --
    (e) -- (C) -- (f) -- (A); 
    \coordinate (g) at (5,0);
    \draw [fill] (5,0) circle [radius=0.1];
  \end{tikzpicture}
  \caption{An arrangement of three lines in the real plane, and its
  poset $\FF$ of faces.}
  \label{fig:11}
\end{figure}
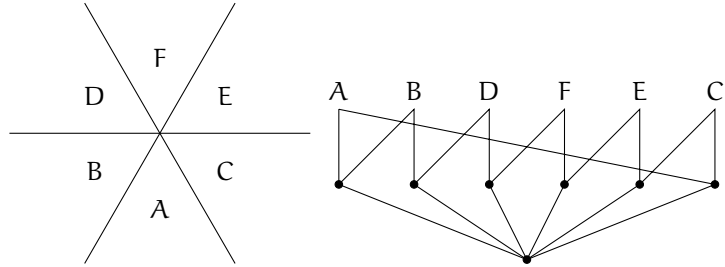

\begin{remark}
  An oriented matroid \FFF\ is uniquely determined by its covectors,
  and also by several other equivalent combinatorial systems, e.g.,
  vectors, basis signatures, or the set of topes. The oriented matroid
  \FFF\ is considered to consist of any and all of these notions - see
  \cite{BLSWZ}. Notice that knowledge of the adjacency relation among
  topes (i.e. the `tope graph' \cite[Definition 4.2.1]{BLSWZ}) is
  enough to reconstruct the oriented matroid up to a reorientation
  (i.e. a global change of sign in a component of the vectors).
\end{remark}


\begin{definition}[Composition of sign vectors]
   Given two sign vectors $X,Y\in \{+,0,-\}^E$ define a sign
   vector $X\circ Y$ as
   $$
   (X\circ Y )_e =
   \left\{\begin{array}{ll} Y_e & \textrm{if } X_e =0 \\
       X_e & \textrm{else.} \end{array}\right.
   $$
\end{definition}

\begin{remark}
     If $X$ and $Y$ are covectors of an oriented matroid (and thus
     correspond to cells on the sphere), then $X\circ
     Y$ correspond to the cell obtained by `moving slightly off $X$ towards $Y$'.
\end{remark}

We can now introduce the topological object on which we'll focus.

\begin{definition}
  The {\em Salvetti poset} of of the given oriented matroid is the set $$\SS=\{(F,C) \in \FF\times\TT\mid
F\circ C = C \}$$ ordered by $(F,C)\leq (F',C')$ if $F\geq F'$ and
$F\circ C' = C$.
\end{definition}

\begin{remark}[Arrangements of hyperplanes]
  In the particular case where the arrangement $\mathscr A$ of Definition
  \ref{def:pseudoarr} is  induced by the intersection of linear hyperplanes with
  the unit sphere,  Salvetti proved \cite{Salvetti} that
  $\vert \SS \vert$ can be embedded as a deformation retract into the complement of the
  complexification of the hyperplanes.
\label{rem:arr}
\end{remark}

\begin{definition}[Definition 4.2.9 of \cite{BLSWZ}]
  Let $M$ be a given oriented matroid with set $\FF$ of covectors and
  set $\TT$ of topes.  Given $B\in \TT$ let $\TT_B$ denote the poset
  of all topes ordered by
$$
T\leqtope{B} R \Leftrightarrow S(B,T)\subseteq S(B,R)
$$
 where the {\em separating set} $S(X,Y)$ of two sign vectors $X,Y\in
 \{+,0,-\}^E$ is defined as $S(X,Y):=\{e\in E \mid X_e=-Y_e\neq 0\}$.
\end{definition}

\begin{remark}\label{rem:intervals}
  The interval determined by $R \leqtope{B}T$ in $\TT$ will be denoted
  $\pos{R}{T}$. Note that it does not depend on $B$, as long as
  $S(T,R)\cap S(B,R) =\emptyset$.
\end{remark}

For the purposes of what follows we need to replace
$\vert \SS \vert$ with another, homotopy equivalent simplicial complex.

\begin{definition}
Let  $\QQ:=(\TT\times\TT,\leq)$ be the poset given on the set $\TT\times\TT$ by the order relation
$$
(T,R) \leq (T',R') :\Leftrightarrow T\leqtope{T'} R \leqtope{T'} R' 
$$
\end{definition}

We show that $\leq$ is transitive, and leave reflexivity and anti-symmetry to the reader.
Let $(T,R) \leq (T',R')$ and $(T',R') \leq
  (T'',R'')$. Then by definition (a) $T\leqtope{T'} R \leqtope{T'} R'$ and
  (b) $T'\leqtope{T''} R' \leqtope{T''} R''$. From (b) follows in
  particular $T'\leqtope{T''}R'$, and this interval has, by Remark
  \ref{rem:intervals}, the same structure in $\TT_{T''}$ as in
  $\TT_{T'}$. Therefore, from (a) we deduce
  $T'\leqtope{T''}T\leqtope{T''}R\leqtope{T''}R'$. With (b), this
  implies $T\leqtope{T''} R \leqtope{T''} R''$, meaning
  $(T,R)\leq(T'',R'')$, as required.

\begin{remark}\label{rem:tope}
The poset \QQ\ can be described in terms of the {\em tope graph}  of \FFF\ \cite[Definition  4.2.1]{BLSWZ}: $(T,R) \leq (T',R')$ if and only if some geodesic from $T$ to $R$ can be extended to a geodesic from $T'$ to $R'$.
\end{remark}

\begin{lemma}\label{lem:homeq}
  The function
$$
\SS \to \QQ; \quad (F,C) \mapsto (C, F\circ (-C)) 
$$
is a poset morphism and induces a homotopy equivalence
$\vert\SS \vert \simeq \vert\QQ \vert$.
\end{lemma}
\begin{proof}
  The given function is clearly order-preserving. Moreover, for any given $(T,R)\in \QQ$ the preimage of $\QQ_{\leq (T,R)}$ is $$\bigg\{[F,F\circ T] \bigg\vert
  \sigma_{\mathscr A}^{-1}(F)\in \bigcap_{\substack{e\not \in S(R,T),
      \\ T\in S_e^{\tau}}}S^{\tau}_e  \bigg\}$$
  which, as a poset, is isomorphic to the poset of those cells in the
  arrangement of pseudospheres that lie in the relative interior of
  the region containing $R$ and delimited by the pseudospheres not
  separating $R$ from $T$. This poset is contractible, e.g. by
   \cite[Proposition 4.2.6 (c)]{BLSWZ} and \cite[Theorem 4.1]{benedetti}, and we conclude
  with Remark \ref{rem:replace}.
\end{proof}

\begin{lemma}\label{rho}
  For $(R,T)\in \QQ$ set $\rho(R,T):=(-T,R)$. This defines a poset isomorphism $\rho: \QQ\to \QQ^{op}$. Moreover, $\rho^4=id$.
\end{lemma}

\begin{remark}\label{3lem}
The following technical facts are a corollary of  \cite[Proposition
4.2.10]{BLSWZ}, and will be used in the proof of Lemma \ref{rho}. For all $A,B,C,D\in \TT$:

   \begin{itemize}
   \item[(a)]
$A\leqtope{B}C \Rightarrow -A\leqtope{-B} -C $;
 \item[(b)] 
$A\leqtope{B} C \leqtope{B} D \Rightarrow C\leqtope{A} D$; 
 \item[(c)] 
$A\leqtope{B} C \Rightarrow B\leqtope{-C} A$.
 \end{itemize}
\end{remark}

\begin{proof}[Proof of Lemma \ref{rho}]
  It is enough to prove that $\rho: \QQ \to \QQ^{op}$ is a poset
  map. To this end let $(R,T)\leq (R',T')\in \QQ$, meaning
  $R'\leqtope{R'} R\leqtope{R'} T\leqtope{R'} T'$. 
 Now: $R\leqtope{R'}T$ implies $R'\leqtope{-T}R$ by Remark \ref{3lem}.(c), while
 from $T\leqtope{R'}T'(\leqtope{R'}-R')$ we get $ T' \leqtope{T} -R'$
 (Remark \ref{3lem}.(b)) and thus  $-T'\leqtope{-T} R'$ (Remark \ref{3lem}.(a)). Together, we obtain $-T'\leqtope{-T} R' \leqtope{-T} R$, i.e., $(-T,R)\geq (-T',R')$ as required.
\end{proof}

\begin{theorem}\label{Th1}
 The assignment $n\mapsto \rho^n$ defines an action of $\mathbb Z_4$
 on $\Delta(\QQ)$ (and thus a simplicial action on the complex $\vert\QQ \vert$).
\end{theorem}

\begin{proof}
  This follows from Lemma \ref{rho} with Remark \ref{rem:op}.
\end{proof}

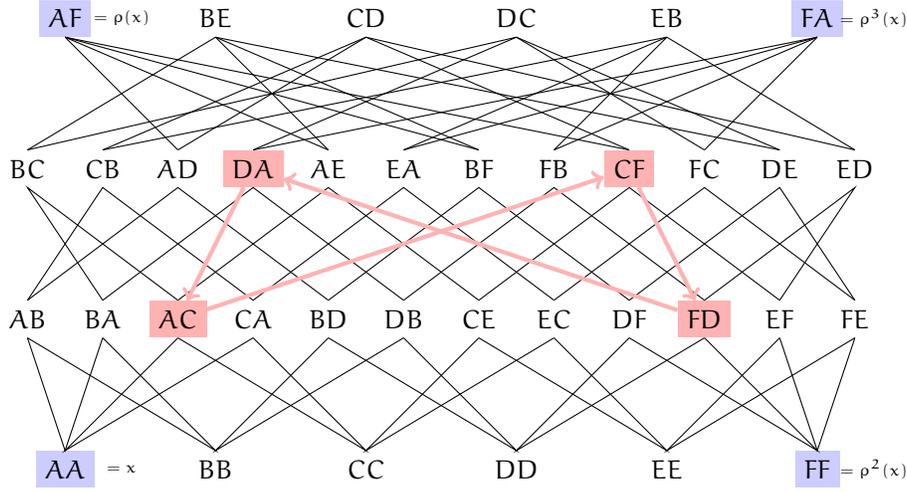
\begin{figure}[h]\centering
  
  \begin{tikzpicture}[scale=0.5]
    \node[fill=blue!20] (AF2) at (1,10) {$AF$}; \node (AF) at (2.5,10)
    {{\tiny $=\rho(x)$}}; \coordinate (AF-) at (1,9.5); \node (BE) at
    (5,10) {$BE$}; \coordinate (BE-) at (5,9.5) {}; \node (CD) at
    (9,10) {$CD$}; \coordinate (CD-) at (9,9.5) {}; \node (DC) at
    (13,10) {$DC$}; \coordinate (DC-) at (13,9.5) {}; \node (EB) at
    (17,10) {$EB$}; \coordinate (EB-) at (17,9.5) {};
    \node[fill=blue!20] (FA) at (21,10) {$FA$}; \node (FA2) at
    (22.5,10) {{\tiny $=\rho^3(x)$}}; \coordinate (FA-) at (21,9.5)
    {};
    \coordinate (AD+) at (4,6.5) {}; \node (AD) at (4,6) {$AD$};
    \coordinate (AD-) at (4,5.5) {}; \coordinate (DA+) at (6,6.5) {};
    \node[fill=red!30] (DA) at (6,6) {$DA$}; \coordinate (DA-) at
    (6,5.5) {};
    \coordinate (BC+) at (0,6.5) {}; \node (BC) at (0,6) {$BC$};
    \coordinate (BC-) at (0,5.5) {}; \coordinate (CB+) at (2,6.5) {};
    \node (CB) at (2,6) {$CB$}; \coordinate (CB-) at (2,5.5) {};
    \coordinate (AE+) at (8,6.5) {}; \node (AE) at (8,6) {$AE$};
    \coordinate (AE-) at (8,5.5) {}; \coordinate (EA+) at (10,6.5) {};
    \node (EA) at (10,6) {$EA$}; \coordinate (EA-) at (10,5.5) {};
    \coordinate (BF+) at (12,6.5) {}; \node (BF) at (12,6) {$BF$};
    \coordinate (BF-) at (12,5.5) {}; \coordinate (FB+) at (14,6.5)
    {}; \node (FB) at (14,6) {$FB$}; \coordinate (FB-) at (14,5.5) {};
    \coordinate (CF+) at (16,6.5) {}; \node[fill=red!30] (CF) at
    (16,6) {$CF$}; \coordinate (CF-) at (16,5.5) {}; \coordinate (FC+)
    at (18,6.5) {}; \node (FC) at (18,6) {$FC$}; \coordinate (FC-) at
    (18,5.5) {}; \coordinate (DE+) at (20,6.5) {}; \node (DE) at
    (20,6) {$DE$}; \coordinate (DE-) at (20,5.5) {}; \coordinate (ED+)
    at (22,6.5) {}; \node (ED) at (22,6) {$ED$}; \coordinate (ED-) at
    (22,5.5) {};
    \coordinate (AB+) at (0,2.5) {}; \node (AB) at (0,2) {$AB$};
    \coordinate (AB-) at (0,1.5) {}; \coordinate (BA+) at (2,2.5) {};
    \node (BA) at (2,2) {$BA$}; \coordinate (BA-) at (2,1.5) {};
    \coordinate (AC+) at (4,2.5) {}; \node[fill=red!30] (AC) at (4,2)
    {$AC$}; \coordinate (AC-) at (4,1.5) {}; \coordinate (CA+) at
    (6,2.5) {}; \node (CA) at (6,2) {$CA$}; \coordinate (CA-) at
    (6,1.5) {}; \coordinate (BD+) at (8,2.5) {}; \node (BD) at (8,2)
    {$BD$}; \coordinate (BD-) at (8,1.5) {}; \coordinate (DB+) at
    (10,2.5) {}; \node (DB) at (10,2) {$DB$}; \coordinate (DB-) at
    (10,1.5) {}; \coordinate (CE+) at (12,2.5) {}; \node (CE) at
    (12,2) {$CE$}; \coordinate (CE-) at (12,1.5) {}; \coordinate (EC+)
    at (14,2.5) {}; \node (EC) at (14,2) {$EC$}; \coordinate (EC-) at
    (14,1.5) {}; \coordinate (DF+) at (16,2.5) {}; \node (DF) at
    (16,2) {$DF$}; \coordinate (DF-) at (16,1.5) {}; \coordinate (FD+)
    at (18,2.5) {}; \node[fill=red!30] (FD) at (18,2) {$FD$};
    \coordinate (FD-) at (18,1.5) {}; \coordinate (EF+) at (20,2.5)
    {}; \node (EF) at (20,2) {$EF$}; \coordinate (EF-) at (20,1.5) {};
    \coordinate (FE+) at (22,2.5) {}; \node (FE) at (22,2) {$FE$};
    \coordinate (FE-) at (22,1.5) {};
    \coordinate (AA+) at (1,-1.5) {}; \node[fill=blue!20] (AA) at
    (1,-2) {$AA$}; \node (AA2) at (2.5,-2) {{\tiny $= x $}};
    \coordinate (BB+) at (5,-1.5) {}; \node (BB) at (5,-2) {$BB$};
    \coordinate (CC+) at (9,-1.5) {}; \node (CC) at (9,-2) {$CC$};
    \coordinate (DD+) at (13,-1.5) {}; \node (DD) at (13,-2) {$DD$};
    \coordinate (EE+) at (17,-1.5) {}; \node (EE) at (17,-2) {$EE$};
    \coordinate (FF+) at (21,-1.5) {}; \node[fill=blue!20] (FF) at
    (21,-2) {$FF$}; \node (FA2) at (22.5,-2) {{\tiny $=\rho^2(x)$}};
    \draw (AF-) -- (AD+); \draw (AF-) -- (AE+); \draw (AF-) -- (BF+);
    \draw (AF-) -- (CF+); \draw (BE-) -- (BF+); \draw (BE-) -- (BC+);
    \draw (BE-) -- (DE+); \draw (BE-) -- (AE+); \draw (CD-) -- (CB+);
    \draw (CD-) -- (CF+); \draw (CD-) -- (AD+); \draw (CD-) -- (ED+);
    \draw (DC-) -- (DE+); \draw (DC-) -- (DA+); \draw (DC-) -- (BC+);
    \draw (DC-) -- (FC+); \draw (EB-) -- (ED+); \draw (EB-) -- (EA+);
    \draw (EB-) -- (FB+); \draw (EB-) -- (CB+); \draw (FA-) -- (FB+);
    \draw (FA-) -- (FC+); \draw (FA-) -- (DA+); \draw (FA-) -- (EA+);
    \draw (AB+) -- (AD-) -- (BD+); \draw (DB+) -- (DA-) -- (BA+);
    \draw (BA+) -- (BC-) -- (AC+); \draw (CA+) -- (CB-) -- (AB+);
    \draw (AC+) -- (AE-) -- (CE+); \draw (EC+) -- (EA-) -- (CA+);
    \draw (BD+) -- (BF-) -- (DF+); \draw (FD+) -- (FB-) -- (DB+);
    \draw (CE+) -- (CF-) -- (EF+); \draw (FE+) -- (FC-) -- (EC+);
    \draw (DF+) -- (DE-) -- (FE+); \draw (EF+) -- (ED-) -- (FD+);
    \draw (AB-) -- (AA+) -- (BA-); \draw (AC-) -- (AA+) -- (CA-);
    \draw (AB-) -- (BB+) -- (BA-); \draw (BD-) -- (BB+) -- (DB-);
    \draw (AC-) -- (CC+) -- (CA-); \draw (EC-) -- (CC+) -- (CE-);
    \draw (BD-) -- (DD+) -- (DB-); \draw (DF-) -- (DD+) -- (FD-);
    \draw (EC-) -- (EE+) -- (CE-); \draw (EF-) -- (EE+) -- (FE-);
    \draw (DF-) -- (FF+) -- (FD-); \draw (EF-) -- (FF+) -- (FE-);
    \draw [->, ultra thick, red!30] (AC) -- (CF); \draw [->, ultra
    thick, red!30] (CF) -- (FD); \draw [->, ultra thick, red!30] (FD)
    -- (DA); \draw [->, ultra thick, red!30] (DA) -- (AC);

  \end{tikzpicture}
  \caption{The poset $\QQ$ for the arrangement of Figure
    \ref{fig:11}. Two orbits of the $\mathbb Z_4$-action are
    shaded. The image of the inclusion of $\SS$ given in Lemma \ref{lem:homeq}
    are all elements of rank $0,1$ and $3$.}
\end{figure}

\begin{remark} Suppose a group $G$ acts on the set of topes of \FFF,
  preserving adjacency. Then $G$ acts on both the Salvetti poset \SS\
  and the tope-pairs poset \QQ\ by order-preserving maps (see
  Remark~\ref{rem:tope}), and the homotopy equivalence of
  Lemma~\ref{lem:homeq} is $G$-equivariant. In case $G$ is a finite
  real linear group generated by reflections, and  \FFF\ is the
  oriented matroid of the associated real reflection arrangement, the
  homotopy equivalence of \SS\ with the complexified complement can be
  chosen to be $G$-equivariant \cite{Salv87}.  Then the orbit space
  $|\QQ|/G$ of the simplicial $G$-action on $|\QQ|$ is homotopy
  equivalent to the complement in $\C^r$ of the $G$-discriminant
  $D_G$. $|\QQ|/G$ is not the realization of a poset, but rather is a
  ``trisp," realizing a small acyclic category \cite{Kozlov}. There is
  another combinatorial model for $\C^r - D_G$ based on the lattice of
  $G$-noncrossing partitions. We ask for a combinatorial homotopy
  equivalence connecting these two models.

\end{remark} 


\section{Combinatorial deconing}

Recall that, throughout, $\FF$ denotes the poset of covectors of an
arbitrary (but fixed) oriented matroid on the
ground set $E$.

\begin{definition}\label{DefDec}
  Every choice of an element $e\in E$ gives rise to an {\em affine
    oriented matroid} with poset of covectors
  \begin{displaymath}
    d_e\FF:=\{F\in \FF\mid F_e = +\}.
  \end{displaymath}
 From now an arbitrary element $e\in E$ will be fixed, and 
 we will simply write $d\FF$.

Accordingly, we define the subposets
\begin{displaymath}
  d\SS:=\{[F,C]\in \SS \mid F,C\in\FF\} \subseteq \SS
\end{displaymath}
\begin{displaymath}
  d\QQ:=\{(R,T)\in \QQ \mid R_e=T_e=+\}\subseteq \QQ
\end{displaymath}
\end{definition}

\begin{remark} The map of Lemma \ref{lem:homeq} restricts to a poset
  map $d\SS \to d\QQ$ which induces homotopy equivalence. 
\end{remark}

\begin{definition}
  Consider the oriented matroid of rank $1$ on the ground set
  $\{e\}$, with sets of covectors and topes $\FF_{{e}}$ and $\TT_{{e}}$. The action of $\mathbb Z_4$ on the associated poset
  $\QQ_{{e}}$ is transitive. Choosing $R\in\TT_{{e}}$ we can identify the elements of $\QQ_{{e}}$ with elements of $\mathbb
  Z_4$ so that, for $i=0,\ldots,3$, $\rho^i(R,R)$ is identified with
  the class $[i]\in\mathbb Z_4$.

\end{definition}

\begin{definition}
  Define a function $\Psi: \Delta(\QQ_e)\times\Delta(d\QQ)^{op}\to \QQ$ so that
$$
\Psi (\{{[i]}\},\omega_1<\cdots < \omega_k) :=\left\{
  \begin{array}{ll}
    \rho^i(\omega_1) & i \textrm{ even}\\
    \rho^i(\omega_k) & i \textrm{ odd} \\
  \end{array}\right.
$$
$$
\Psi(\{{[i]},{[i+1]}\},\omega) := \Psi({[i]},\omega) \vee \Psi({[i+1]},\omega)
$$
\end{definition}

\begin{remark}[Joins in $\QQ$]
 Although $\QQ$ is certainly not a lattice, the `join' in the above
  definition - which should be thought of as `the minimum among all elements
  that are above  both terms' -  is well-defined in the cases we need.
 Indeed,
 w.l.o.g. $\Psi({[i]},\omega)\vee\Psi({[i+1]},\omega)=(A,B)\vee(-D,C)$
 for some $A,B,C,D\in \TT$ with $(A,B)\geq (C,D)$, and one sees that
 the join operation determines the element $(-D,B)$.
\end{remark}

\begin{remark}[Notation]\label{rem:notaQe}
   For ease of notation we will from now omit all brackets when
   referring to elements of  $\QQ_{{e}}$ or
   $\Delta(\QQ_{{e}})$. Thus writing for instance $12$ instead of
   $\{[1],[2]\}\subset \mathbb Z_4$, and obtain the picture shown in
   Figure \ref{fig:1}.
\end{remark}

\begin{figure}[h]\centering
  
  \begin{tikzpicture}
    \node (A) at (0,2) {1}; \node (B) at (0,0) {0}; \node (C) at (2,2)
    {3}; \node (D) at (2,0) {2}; \path[-,font=\scriptsize,>=angle 90]
    (A) edge (B) (A) edge (D) (C) edge (B) (C) edge (D);
  \end{tikzpicture}$\quad\quad\quad$
  \begin{tikzpicture}
    \node (A) at (0,2) {01}; \node (B) at (2,2) {12}; \node (C) at
    (4,2) {23}; \node (D) at (6,2) {03}; \node (E) at (0,0) {0}; \node
    (F) at (2,0) {1}; \node (G) at (4,0) {2}; \node (H) at (6,0) {3};
    \path[-,font=\scriptsize,>=angle 90] (A) edge (E) (B) edge (F) (C)
    edge (G) (D) edge (H) (E) edge (D) (F) edge (A) (G) edge (B) (H)
    edge (C);
  \end{tikzpicture}
  
  \caption{Hasse diagrams of the posets $\QQ_e$ and $\Delta(\QQ_e)$,
    using the notational convention of Remark \ref{rem:notaQe}.}
\label{fig:1}
\end{figure}
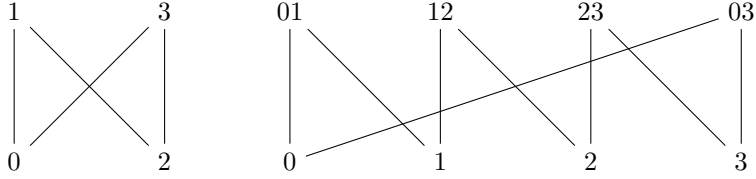

\begin{remark}\label{rem:explicit}
  It will be convenient to examine
  explicitly the function $\Psi$. If
  $\omega=(A_1,B_1)<\ldots<(A_k,B_k)$ is a chain in $d\QQ$, we have
  \begin{displaymath}
    \begin{array}{ll}
    \Psi(0,\omega)=(A_1,B_1);&
    \Psi(01,\omega)=(-B_k,B_1);\\
    \Psi(1,\omega)=(-B_k,A_k);& \Psi(12,\omega)=(-A_1,A_k);\\
    \Psi(2,\omega)=(-A_1,-B_1);&
    \Psi(23,\omega)=(B_k,-B_1);\\
    \Psi(3,\omega)=(B_k,-A_k);& \Psi(03,\omega)=(A_1,-A_k).\\
  \end{array}
  \end{displaymath}
\end{remark}

\begin{lemma}\label{lem:main}
  The function $\Psi$ defines a poset map and induces a homotopy equivalence.
\end{lemma}
\begin{proof} The definition of $\Psi$ shows that it is a poset
  map. To prove homotopy equivalence, 
  We consider preimages of elements $(C,K)\in\QQ$ and verify the condition of Lemma \ref{lem:quillen}.
  \begin{list}{}
    {\setlength{\leftmargin}{2em}
      \setlength{\labelwidth}{0em}
      \setlength{\labelsep}{0em}
      \setlength{\itemindent}{-\leftmargin}
    }
  \item {\em Case 1: $C_e=K_e=+$.} First, from the explicit description of
    $\Psi$ in Remark \ref{rem:explicit} notice the poset isomorphism
    \begin{displaymath}
      \Psi^{-1}(\QQ_{\geq  (C,K)}) 
      \cong \{
      \omega \in \Delta(d\QQ)^{op} \mid
      \max\omega\in d\QQ_{\geq (C,K)}\}.
    \end{displaymath}
    Define a diagram of posets
    \begin{displaymath}
      \DD: (d\QQ_{\geq (C,K)})^{op} \to \operatorname{Pos}; \quad \DD(X,Y)= \Delta^\dag(\QQ_{\leq (X,Y)})^{op}
    \end{displaymath}
    with diagram maps being inclusions.

\def\plim{\operatorname{plim}}
    Then the poset limit $\plim \DD$ is a poset with elements $((X,Y),
    \omega)$, where $(X,Y)\geq (C,K)$ and $\max \omega=(X,Y)$, ordered according to 
    $$
    ((X,Y),\omega) \leq ((X',Y'),\omega') \Leftrightarrow
    (X,Y) \geq (X',Y'),\textrm{ and } \omega \supseteq \omega'.
    $$
    Thus we have an evident poset isomorphism
   $
   \Psi^{-1}(\QQ_{\geq (C,K)}) \cong \plim \DD
   $
\def\hocolim{\operatorname{hocolim}}
   and homotopy equivalences (see \cite[Corollary 2.11]{BK} and
   \cite[Section 15]{Kozlov})
   $$
   \vert  \Psi^{-1}(\QQ_{\geq (C,K)})\vert \simeq \vert\plim \DD \vert
   \simeq \hocolim \vert \DD \vert.
   $$
 Here $\vert \DD \vert$ is the diagram of geometric realizations of \DD\ in the category of topological spaces and continuous maps.
   Now with Lemma \ref{lem:delta} we have that every space $\vert
   \DD(X,Y) \vert$ is homotopy equivalent to $\vert\Delta(\QQ_{\leq
     (X,Y)})^{op}\vert$, thus contractible. With \cite[Theorem
   15.19]{Kozlov} we obtain  
   $$
  \vert \Psi^{-1}(\QQ_{\geq (C,K)})\vert \simeq \hocolim \vert \DD\vert
   \simeq \vert d\QQ_{\geq (C,K)}\vert \simeq *.
   $$
\item {\em Case 2: $-C_e=K_e=+$.} Again, with Remark
  \ref{rem:explicit} we can write explicitly 
\def\PP{\mathcal P}
\def\RR{\mathcal R}
  \begin{align*}
    &\Psi^{-1}(\QQ_{\geq (C,K)}) =\\ &\quad \{(01,
    (A_1,K')<\ldots<(A_k,-C'))\mid (K',-C')\leq(K,-C)\} =:\PP_I\\
    &\quad\cup \{(1,
    (A_1,B_1)<\ldots<(K',-C'))\mid (K',-C')\leq(K,-C)\} =:\PP_{II}\\
    &\quad\cup \{(12,
    (-C',B_1)<\ldots<(K,B_k))\mid (K',-C')\leq(K,-C)\} =:\PP_{III}\\
  \end{align*}
  It is immediate to see that $\PP_{II}=\{1\}\times \Delta
  (\QQ_{\leq (K,-C)})^{op}$ and is thus contractible. Moreover, notice that
  $(1,\omega)\in \PP_{II}$ implies both $(01,\omega) \in
  \PP_I$ and $(12,\omega)\in \PP_{III}$, for all $\omega$. Thus,
  by defining $\RR:=\Delta(\QQ_e)_{\geq \{1\}}\times \PP_{II}$, we
  have a covering of $\Psi^{-1}(C,K)$ by three posets
  $\PP_I,\RR,\PP_{III}$ with $\PP_I\cap \RR \simeq \PP_{III}\cap \RR
  \simeq \PP_{II}$ (thus contractible) and $\PP_I\cap \PP_{III} =
  \emptyset$. By the generalized nerve lemma  \cite[Theorem
  15.24]{Kozlov} applied to the covering of
  $\vert\Psi^{-1}(C,K)\vert$ by its subcomplexes $\vert\PP_I\vert$,
  $\vert\RR\vert$  and
  $\vert\PP_{III}\vert$, the poset
  $\Psi^{-1}(C,K)$ is contractible if $\PP_I$ and $\PP_{III}$ are.

  We are thus left with proving contractibility of $\PP_I$
  (contractibility of $\PP_{III}$ follows by a similar argument). To
  this end, notice first of all that $(A_1,K')<(A_2,B_2)<\ldots
  <(A_k,-C')$ is a chain if and only if
  \begin{displaymath}
    A_k \leqtope{C'} \ldots \leqtope{C'} A_1 \leqtope{C'} K'
    \leqtope{C'} B_2 \leqtope{C'} \ldots \leqtope{C'} B_{k-1}
  \end{displaymath}
  We thus obtain a bijection
  \begin{displaymath}
    \PP_{I} \to \Delta\pos{C'}{K'} \times \Delta^{\dag\dag}\pos{-C'}{K'};\quad\quad (01,\omega) \mapsto \omega
  \end{displaymath}
  which is clearly order-reversing. Thus
  \begin{displaymath}
    \PP_I \simeq \Delta\pos{C'}{K'} \times \Delta^{\dag\dag}\pos{-C'}{K'} \simeq \pos{C'}{K' } \times \Delta\pos{-C'}{K'} \simeq *.
  \end{displaymath}
\item The other cases are treated analogously to the above.
  \end{list}
\end{proof}

\begin{theorem}\label{Th2}
  For every given oriented matroid and any element of its ground set:
$$
\vert \SS \vert \simeq S^1\times \vert d\SS \vert.
$$
\end{theorem}
\begin{proof}
Immediate applying Remark \ref{rem:can} to Lemma \ref{lem:main}.
\end{proof}

\begin{corollary}[Theorem 4.2 of \cite{cordovil}] 
For every given oriented matroid and any element of its ground set:
$\pi_1(\vert \SS \vert) \simeq  \mathbb Z \times \pi_1(\vert d\SS \vert)$.
 \end{corollary}

\section{Non-realizable groups}

We close by exhibiting an oriented matroid \FFF\ for which the fundamental group $\pi_1(|\QQ|) \cong \pi_1(|\SS|)$ is not isomorphic to the fundamental group of the complement of any arrangement (complexified or not) of linear hyperplanes in $\C^r$. Thus the homotopy type of \QQ\ is not represented by a complex arrangement complement. To our knowledge no example of either phenomenon has appeared in the literature. The example illustrates that results such as ours extending properties of arrangement groups to the non-realizable case are strict generalizations of the existing theory.

Let \FFF\ be an oriented matroid on ground set $E$. Let $A_\Z=A_\Z(\FFF)$ be the cohomology ring $H^*(|\SS|,\Z)$ of $|\SS|$. By \cite{GR,BZ}, $A_\Z$ is isomorphic as a graded algebra to the Orlik-Solomon (OS) algebra of the underlying unoriented matroid $\uF$  of \FFF,  the quotient of the exterior algebra on $E$ by the ideal generated by ``boundaries" of circuits in $\uF$ - see \cite{Yuz01}. In particular, $A_\Z$ is generated in degree one, and $A_Z^1 \cong \Z^E$. In fact $\A_Z$ is a free \Z-module, by \cite{JT2}. 

We piece together several known results about the relation of $A$ to $\pi_1(|\SS|)$, and about degree-one resonance varieties of OS algebras, to draw our conclusions. 

\begin{definition}
Let $A$ be a graded $\kk$-algebra, \kk\ a field. The {\em degree-one resonance variety} of $A$ is the subset $\RR^1(A)$ of $A^1$ given by 
\[
\RR^1(A)=\{ a \in A^1 \mid ab=0 \ \text{for some} \ b \in A^1 - \kk a\}.
\]
\end{definition}

We will be concerned with $\RR^1(A)$ for the OS algebra $A=A_\Z \otimes \C$ with complex coefficients. Resonance varieties were introduced in this context and their basic properties established in \cite{Fa97} - see also \cite{Fa07,FY07}. In this case, $\RR^1(A)$ is a union of linear subspaces of $A^1\cong \C^E$, by \cite{LY00,CS3}, every two of which intersect trivially \cite{LY00}. All components are contained in the diagonal hyperplane $H_0=\{x \in \C^E \mid \sum_{e\in E} x_e = 0\}$  \cite[Proposition  2.1]{Yuz95}. Each rank-two flat $X$ of cardinality $|X| \geq 3$ in $\uF$ gives rise to a component $L_X$ of $\RR^1(A)$ of dimension $|X|-1$, called a {\em local component}, and defined by 
\[
\ell_X=\{ x \in H_0 \mid x_e=0 \ \text{for} \ e \not\in X\}.
\]
Non-local components of $\RR^1(A)$ arise from multinets supported on rank-three submatroids of $\uF$ \cite{FY07}. The smallest such multinet is supported on the six-point graphic matroid $\uF(K_4)$. We will see below that the linear isomorphism type of $\RR^1(A)$ is an invariant of the fundamental group $\pi_1(|\QQ|)$.

Let $A^{\leq 2}=A/\oplus_{p\geq 3} A^p$ denote the truncation of $A$ to degree two. Clearly $\RR^1(A) \cong \RR^1(A^{\leq 2})$. Moreover, 
\begin{lemma} The graded algebra $A^{\leq 2}$ is an invariant of $\pi_1(|\QQ|)$.
\label{lem:char}
\end{lemma}

\begin{proof}
By an argument of \cite{R2} as generalized in \cite[Proposition  1.6]{MatSuc1},  $A^{\leq 2}$ is isomorphic to the degree-two truncation of the cohomology of the group $\pi_1(|\SS|)$, with complex coefficients; this depends only on the fact that $A_\Z \cong H^*(|\QQ|,\Z)$ is a free \Z-module and is generated in degree one. Then $A^{\leq 2}$ is determined up to graded-algebra isomorphism by $\pi_1(|\QQ|)$. \end{proof}

\begin{corollary} The linear isomorphism type of $\RR^1(A)$ is determined by $\pi_1(|\QQ|)$.
\label{cor:inv}
\end{corollary}

The set $\CC=\CC_\FFF$ of components of $\RR^1(A)$ determines a configuration of linear spaces, consisting of the sums $L_S=\sum_{\ell \in S} \ell$ for $S \subseteq \CC$. The associated {\em polymatroid} is the dimension function $d_\CC \colon 2^\CC -\{\emptyset\} \to \Z_{\geq 0}$ defined by $d_\CC(S)=\dim(L_S)$. We say a subset $S$ of $\CC$ is {\em closed} if $d_\CC(S)<d_\CC(T)$ for all $T \supsetneq S$.  A linear isomorphism of $\RR^1(A)$ to $\RR^1(A')$ induces a bijection $\CC \to \CC'$ which preserves closed sets and commutes with $d_\CC$.

Define the {\em support} $\supp(S)$ of $S \subseteq \CC$ by 
\[
\supp(S)=\{e \in E \mid x_e \neq 0 \ \text{for some} \ x \in L_S\}.
\]
One can identify the submatroid $\supp(S)$ of \uF\ for small matroids and small sets $S$. If $|\supp(S)|\leq 9$ and $S$ contains a non-local component $\ell$, then $\supp(\ell)$ contains the six-point matroid with four three-point lines, or one of two nine-point matroids with nine three-point lines. This is easy to check using \cite{Fa97} and \cite{FY07}. One computes in any of these three cases that, if $S$ contains at least three elements, and $S$ is closed, then $|S|\geq 4$. 

The {\em rank-three whirl} is the matroid $\W^3$ on $\{1,2,3,4,5,6\}$ with lines $123$, $345$, and $156$.

\begin{lemma} Suppose $|E|\leq 9$, $S \subseteq \CC$ is closed with $|S|=3$, $d(\ell)=2$ for all $\ell \in S$, and $d(S)=5$. Then $\supp(S)$ is a six-point submatroid of \uF\ isomorphic to $\W^3$.
\label{lem:whirl}
\end{lemma}

\begin{proof} If $|\supp(S)|\leq 5$ then $S$ contains no non-local components, and at most two local components. Thus $|\supp(S)|\geq 6$. By the preceding remarks $S$ contains only local components, which arise from three three-point lines in \uF. 
One checks that $d(S)=6$ unless $\supp(S)$ has six points. The only rank-three matroid on six points with three three-point lines, which is not $\uF(K_4)$, is $\W^3$.
\end{proof}

%

%
%
%
%
%

Now let \FFF\ be the oriented matroid of the non-Pappus arrangement of pseudo-lines \cite{Grun72}.  The underlying rank-three matroid \uF\ has nine points, which we identify with the numbers $1, \ldots, 9$, and eight nontrivial lines
\[
123, 157, 168, 247, 269, 348, 359, 456.
\] 
In any point configuration over a field with these nontrivial lines, $789$ is also a line by Pappus' Theorem.  Thus \uF, the non-Pappus matroid, is not realizable over any field. Let $\QQ=\QQ_\FFF$ be the associated tope-pairs poset. 

\begin{theorem} $\pi_1(|\QQ|)$ is not isomorphic to the fundamental group of the complement of any arrangement of linear hyperplanes in $\C^r$.
\end{theorem}

\begin{proof} Suppose \A\ is an arrangement of linear hyperplanes in $\C^r$, and $\pi_1(|\QQ|)$ is isomorphic to the fundamental group of the complement of \A. Let $\uF'$ denote the underlying matroid, and $A'$ the Orlik-Solomon algebra of \A. By Corollary~\ref{cor:inv} there is a linear isomorphism from $A^1$ to $(A')^1$ carrying $\RR^1(A)$ to $\RR^1(A')$. Then $\uF'$ is a matroid on nine points. We claim $\uF'$ is isomorphic to \uF, hence is not realizable over \C, a contradiction.

To prove the claim, we calculate that the set \CC\ of components of $\RR^1(A)$ consists of eight local components, of dimension two, coming from the eight three-point lines. Then the set $\CC'$ of components of $\RR^1(A')$ also consists of eight two-dimensional subspaces. There is no $S \subseteq \CC$ with $|S|=5$ and $d(S)=5$. Then the same holds for $\CC'$.  Together with the fact that $|\CC'|<9$, this implies $\CC'$ contains no non-local components.

The submatroid of \uF\ on $\{1,2,3,5,7,9\}$ is the rank-three whirl, yielding the set $S=\{\ell_{123},\ell_{157},\ell_{359}\}\subseteq \CC$ with $d(S)=5$. These three components map to three components of $\RR^1(A')$ whose support $\W$ in $\uF'$ is also the rank-three whirl, by Lemma~\ref{lem:whirl}. The image of $\ell_{456}$ is a local component, hence is supported on a three-point line. This line in $\uF'$ meets $\W$ in one point, because $d(S \cup \{\ell_{456}\})=7$. Since $245679$ is also a whirl, one concludes that the eight points in the support of the image of $\{\ell_{123},\ell_{157},\ell_{359},\ell_{456}\}$ form a copy of $\uF-\{8\}$. Finally, since $145678$ and $123478$ are also whirls, the support of the image of $\ell_{168}$ yields the ninth point in $\uF'$, and completes a copy of \uF.
\end{proof}

\begin{ack} This work was mostly done while the second author was
  visiting the institute for Algebra, Geometry, Topology, and their
  Applications (ALTA) at the University of Bremen, and the first author was
  a postdoc there. We gratefully acknowledge the support of the Institute and University, and thank Eva-Maria Feichtner and Dmitry Feichtner-Kozlov for their support and hospitality, and for many useful discussions.
\end{ack}
\bibliographystyle{abbrv}
\bibliography{circlebib}
\end{document}